\documentclass[11pt]{article}
\usepackage{amsmath, comment}
\usepackage{amsthm}
\usepackage{amssymb}
\usepackage{amsfonts}

\usepackage{authblk}
\usepackage{dsfont}
\usepackage{color}

\newcommand{\dd}{\mathrm{d}}
\newcommand{\R}{\mathbb{R}}

\newcommand{\E}{\mathbf{E}}
\newcommand{\p}{\mathbf{P}}

\newcommand{\bone}{\mathds 1}
\newcommand{\dint}{\int \hspace*{-5pt} \int}
\newcommand{\trint}{\int \hspace*{-5pt} \int \hspace*{-5pt} \int}

\DeclareMathOperator{\Leb}{Leb}

\theoremstyle{plain}
\newtheorem{lemma}{Lemma}
\newtheorem{theorem}{Theorem}

\newtheorem{corollary}{Corollary}

\theoremstyle{remark}

\begin{document}

\title{Path properties of L\'evy driven mixed moving average processes\thanks{Dedicated to Professor Nikolai Leonenko on the occasion of his 75th birthday.}}

\author[1]{Danijel Grahovac\thanks{dgrahova@mathos.hr}}
\author[2]{P\'eter Kevei\thanks{kevei@math.u-szeged.hu}}
\author[3]{Orimar Sauri\thanks{osauri@math.aau.dk}}
\affil[1]{School of Applied Mathematics and Informatics, J. J. Strossmayer University of Osijek, Croatia}
\affil[2]{Bolyai Institute, University of Szeged, Hungary}
\affil[3]{Department of Mathematical Sciences, Aalborg University, Denmark}
\date{}

\maketitle

\begin{abstract}
We derive general sufficient conditions for the existence of c\`adl\`ag and continuous modifications of L\'evy-driven mixed moving average processes. The conditions are explicit and easy to verify and applied to supOU, well-balanced supOU, trawl, and power-weighted supOU processes. In these examples, the conditions are shown to be close to optimal.


\end{abstract}	
	
\section{Introduction and main results}

Consider the \textit{mixed moving average} (MMA) process
\begin{equation}\label{eq:def_X}
	X(t) = \dint_{ V \times \R} f(x,t-u) \Lambda(\dd x, \dd u),\,\,t\in\R,
\end{equation}
where $V$ is a topological space with Borel $\sigma$-algebra $\mathcal{B}(V)$, 
$f: V \times \R \to [0,\infty)$ is a deterministic nonnegative measurable function and $\Lambda$ is an infinitely divisible independently scattered 
random measure, given by  
\begin{equation}\label{eq:Lambda}
	\begin{split}
	\Lambda(\dd x, \dd u) &= m \pi(\dd x) \dd u + 
	\int_{(0,1]} z (\mu - \nu)(\dd x, \dd u , \dd z)\\
	& \quad + \int_{(1,\infty)} z \mu(\dd x, \dd u , \dd z).
	\end{split}
\end{equation}
Here, $m \in \R$,
$\mu$ is a Poisson random measure with intensity
$\nu(\dd x, \dd u, \dd z) = \pi(\dd x) \dd u \lambda(\dd z)$, where $\pi$ is a measure on $V$ and $\lambda$ is a L\'evy measure on $(0,\infty)$. This class of random measures is often referred to as \textit{L\'evy bases}, see \cite{barndorff2018ambit} for more details.  
For simplicity, we assume that $f$ is nonnegative and that 
$\lambda$ is concentrated on $(0,\infty)$. Neither of these assumptions
is essential, and the corresponding results in the general case can be 
easily formulated.

The necessary and sufficient conditions for the existence of the integral in 
\eqref{eq:def_X} are given in \cite[Theorem 2.7]{rajputrosinski1989}, see also \cite[Chapter 5]{barndorff2018ambit}. In what follows, we assume that the integral in 
\eqref{eq:def_X} is well-defined. 
Observe that MMA processes are infinitely divisible (ID) and strictly stationary. Note that classical  L\'evy-driven ID moving average processes arise 
when $V $ is a singleton. Well-known specific members of the MMA class are  
superpositions of Ornstein--Uhlenbeck-type (\textit{supOU}) processes and \textit{trawl} processes.

There is a vast, recent literature on these processes in both applied and 
theoretical probability. 
SupOU models have been successfully used in finance to model stochastic volatility, see e.g.,~\cite{curato2019}, and the references therein. 
Besides applications in finance, integer valued trawl models are 
frequently used in queueing theory \cite{barndorff2014intTrawl}.
Extremal behavior was investigated in \cite{rm2022} 
for L\'evy-driven $d$-dimensional moving average processes, and 
in \cite{Fasen05} for the MMA process with $V = \R$, both with heavy-tailed
jump measure $\lambda$.
Distributional limit theorems were studied in 
\cite{GLT19, GLT2020multifaceted, GLT21} for supOU processes,
and in \cite{pakkanen2021, talarczyk2020} for trawl processes.
In \cite{GK2,GK}, the almost sure growth rates of
supOU and integrated supOU processes were studied.

To the best of our knowledge, there are only a few results on the path properties 
of MMA processes. For symmetric $\alpha$-stable driving measure 
with $\alpha \in (1,2)$, sufficient conditions for the existence of a c\`adl\`ag modification were obtained in 
\cite[Theorem 4.3]{BasseRos} for general kernel, and in \cite{basse2020sufficient} for supOU processes.
In \cite{BasseRos2013b}, both sufficient conditions
and necessary conditions were provided for the existence of absolutely 
continuous sample paths. Sufficient conditions for certain sample path properties of ID processes were obtained in \cite{rosinski1989}. In the present paper, we obtain conditions ensuring that $X$ has a c\`adl\`ag modification.

The main motivation for this work is Proposition 1 in \cite{GK}, where it was shown 
that under natural conditions, the supremum of a supOU process on 
any nonempty open interval is almost surely infinite. 
In particular, there exists no c\`adl\`ag modification.
Such path behavior was obtained in \cite{Maejima} for linear fractional
stable motion and in \cite[Theorem 3.7]{CDH} for the solution of the stochastic
heat equation.

Our main result is the following general sufficient condition for the 
existence of a c\`adl\`ag modification. We combine Billingsley's classical 
criteria (\cite[Theorem 13.6]{Billingsley}) with a moment bound.
In what follows, $\Leb$ stands for the Lebesgue measure on $\R$,
$a \wedge b = \min \{ a, b\}$, and $a \vee b = \max \{ a, b \}$.

\begin{theorem} \label{thm:cadlag}
	Let $f:V\times\R\to [0,\infty)$ be a measurable function such that $f(x, \cdot)$ is c\`adl\`ag for each $x \in V$, and $X$ in \eqref{eq:def_X} is well-defined. Assume that there exist a Borel set $A \subset V \times \R$ and a function 
	$g$, such that $(\pi \times \Leb) (A) < \infty$,  
	for all $(x,-u) \not\in A$
	\begin{equation} \label{eq:f-ass-bound}
		\sup_{t \in [0,1] }  f(x,t-u) \leq g(x, -u), 
	\end{equation}
	and
	\begin{equation} \label{eq:f-g-ass}
		\trint_{A^c \times (1,\infty)} 
		( g(x,-u) z  \wedge 1) \, \pi(\dd x) \dd u \lambda(\dd z) <\infty.
	\end{equation}
	Suppose, in addition, that there exist $\alpha \in [1,2]$, $C > 0$, 
	and $\varepsilon > 0$ such that
	$\int_{(0,1]} z^\alpha \lambda(\dd z) < \infty$, and for all $0 \leq s < t $
	\begin{equation} \label{eq:f-int-ass1_mma}
		\dint_{V \times \R} \left| f(x,t-u) - f(x,-u) \right|^{\alpha} 
		\pi(\dd x) \dd u \leq C t^{1/2+\varepsilon/2}, 
	\end{equation}
	and
	\begin{equation} \label{eq:f-int-ass2_mma}
		\dint_{V \times \R} 
		| f(x,t-u) - f(x, s-u)|^\alpha  \wedge 
		| f(x,s-u) - f(x, -u)|^\alpha  \pi(\dd x) \dd u
		\leq C t^{1+\varepsilon}.
	\end{equation}
	Then $X$ has a c\`adl\`ag modification.
\end{theorem}

When the kernel $f$ is continuous, we obtain a simpler criterion 
for the existence of a continuous modification.

\begin{theorem} \label{thm:cont}
	Let $f: V \times \R \to [0,\infty)$ be a measurable function such that
	$f(x, \cdot)$ is continuous for each $x \in V$, and $X$ in \eqref{eq:def_X} is well-defined. Assume that 
	there exist a Borel set $A \subset V \times \R$ and a function 
	$g$, such that $(\pi \times \mathrm{Leb}) (A) < \infty$, 
	\eqref{eq:f-ass-bound} holds for all $(x,-u) \not\in A$, and 
	\eqref{eq:f-g-ass} holds.
	Assume further that there exist $\alpha \in [1,2]$, $C > 0$, 
	and $\varepsilon > 0$ such that
	$\int_{(0,1]} z^\alpha \lambda(\dd z) < \infty$, and
	for $0 \leq s < t $
	\begin{equation} \label{eq:cont-ass}
		\dint_{V \times \R} | f(x,t-u) - f(x, s-u)|^\alpha \pi(\dd x) \dd u
		\leq C t^{1+\varepsilon}.
	\end{equation}
	Then $X$ has a continuous modification.
\end{theorem}

The rest of the paper is organized as follows. In Section \ref{sect:examples},
we spell out the general theorem for specific examples. We also 
demonstrate that our conditions are close to optimal by
showing that, if some are omitted, then no c\`adl\`ag modification exists.
The proofs of the main results are in Section \ref{sect:proof}.

\section{Examples} \label{sect:examples}

\subsection{SupOU processes}

Let $V=(0,\infty)$ and $f(x,s)= e^{-xs}\bone(s \geq 0)$. 
Assume that 
\begin{equation} \label{eq:supOU-cond}
	\int_{(0,\infty)} x^{-1} \pi(\dd x)<\infty  \quad \text{and } \
	\int_{(1,\infty)} \log z \, \lambda(\dd z)<\infty,
\end{equation}
which are the 
necessary and sufficient conditions for the existence of the 
integral in \eqref{eq:supOU-def} below. 
The superposition of Ornstein--Uhlenbeck type processes (\textit{supOU}) is defined as
\begin{equation} \label{eq:supOU-def}
	X(t) = \dint_{(0,\infty) \times (-\infty, t]} e^{-x(t-s)} \Lambda(\dd x, \dd s).
\end{equation}

SupOU processes 
were introduced by Barndorff-Nielsen \cite{Barn} 
with a slightly different parametrization.
The marginal distribution and the dependence 
structure of these processes can be chosen separately, 
allowing for a flexible model (see \cite{bnleonenko2005} for examples). 
SupOU processes have an interesting and rich limiting behavior.
Depending on the L\'evy measure $\lambda$ and the dependence measure 
$\pi$, the limiting process of the centered and scaled
integrated supOU process can be
Brownian motion, fractional Brownian motion, stable L\'evy process,
or stable process with dependent increments, see \cite{GLT19, GLT2020multifaceted, GLT21}.

To the best of our knowledge, there are only few results on 
the existence of a c\`adl\`ag modification of such processes.
If the underlying L\'evy basis has finite variation, that is
$\int_{(0,1]} z \lambda(\dd z) < \infty$, then 
Theorem 3.12 in \cite{barndorff2011multivariate} gives
conditions in the multivariate setup for the existence of a c\`adl\`ag modification
which has finite variation on compact sets. 
For symmetric $\alpha$-stable L\'evy bases with $\alpha \in (1,2)$, 
the existence  of a  c\`adl\`ag modification was obtained in
\cite{basse2020sufficient}.

Note that for the existence of the process, $\pi$ does not need to be a 
finite measure. On the other hand, this is assumed in most of the 
literature, among others in 
\cite{barndorff2011multivariate, basse2020sufficient}.

\begin{corollary} \label{corr:supOU}
	Let $\lambda$ be a L\'evy measure on $(0,\infty)$, 
	and $\pi$ be a measure on $(0,\infty)$ such that 
	\eqref{eq:supOU-cond} holds. If for  
	some $\varepsilon > 0$
	\begin{equation} \label{eq:supOU-int-ass}
		\int_{(1,\infty)} x^\varepsilon \pi(\dd x) < \infty, 
	\end{equation}
	then the supOU process $X$ in \eqref{eq:supOU-def} 
	has a c\`adl\`ag modification.
\end{corollary}

\begin{proof}
	First, we show that conditions \eqref{eq:f-ass-bound} and \eqref{eq:f-g-ass} are satisfied with $A = \emptyset$ and
	\[
	g(x,s) =
	\begin{cases}
		e^{-xs}, & s \geq 0, \\
		1, & s \in [-1,0], \\
		0, & s < -1.
	\end{cases}
	\]
	Indeed, simple calculations show that \eqref{eq:f-ass-bound} holds, and
	\[
	\begin{split}
		& \trint_{(0,\infty) \times \R \times (1,\infty)} 
		( g(x,s) z \wedge 1) \, \pi(\dd x) \dd s \lambda(\dd z) \\
		& = \pi((0,\infty)) \overline \lambda(1) + 
		\int_{(0,\infty)} x^{-1} \pi(\dd x) \int_{(1,\infty)} ( 1 + \log z) \lambda(\dd z)  
		< \infty,
	\end{split}
	\]
	proving \eqref{eq:f-g-ass}. Here, and later on 
	$\overline \lambda(r) = \lambda((r,\infty))$.
	
	Next, we show that \eqref{eq:f-int-ass1_mma} and \eqref{eq:f-int-ass2_mma}
	are satisfied with $\alpha = 2$. We have for $u \leq s \leq t$
	\begin{equation} \label{eq:supOU-aux1}
		f(x,t-u) -f(x,s-u) = e^{x(u-s)} ( e^{-x(t-s)} -1). 
	\end{equation}
	Therefore, for $\varepsilon \in (0,1]$
	\begin{equation} \label{eq:supOU-int-bound}
		\begin{split}
			& \dint_{(0,\infty) \times (-\infty, 0]} 
			(f(x,t-u) -f(x,-u))^2 \pi(\dd x) \dd u \\
			& = 
			\int_{(0,\infty)} \frac{1}{2x} ( 1 - e^{-x t} )^2 \pi(\dd x) \\
			& \leq \int_{(0,\infty)} \frac{1}{2x} \left( (xt)^{1 + \varepsilon}
			\bone(xt \leq 1) + \bone(xt > 1) \right) \pi(\dd x) \\
			& \leq t^{1+\varepsilon} \frac{1}{2} \int_{(0,\infty)} x^\varepsilon \pi(\dd x),
		\end{split}
	\end{equation}
	and similarly, 
	\[
	\dint_{(0,\infty) \times (0,t]} f(x,t-u)^2 \, \pi(\dd x) \dd u 
	\leq t \pi((0,\infty)).
	\]
	Thus, \eqref{eq:f-int-ass1_mma} holds. To see \eqref{eq:f-int-ass2_mma}, note 
	that by \eqref{eq:supOU-int-bound}
	\[
	\begin{split}
		\dint_{(0,\infty) \times (-\infty, 0]} &
		\left[ 
		( f(x,t-u) - f(x, s-u))^2  \right. \\
		& \qquad \left. \wedge (f(x,s-u) - f(x, -u))^2 \right] \pi(\dd x) \dd u 
		\leq c t^{1+\varepsilon}. 
	\end{split}
	\]
	Furthermore, $f(x,s-u) - f(x,-u) = 0$ for $u > s$, so it remains to 
	handle the integral on $(0,s]$.
	Using again \eqref{eq:supOU-aux1}, a similar argument as in 
	\eqref{eq:supOU-int-bound} gives
	\[
	\begin{split}
		& \dint_{(0,\infty) \times (0,s]} 
		(f(x,t-u) -f(x,s-u))^2 \, \pi(\dd x) \dd u \\
		& = \int_{(0,\infty)}  \int_0^s (1-e^{-x(t-s)})^2  e^{-2xu} 
		\, \dd u \pi(\dd x) \\
		& \leq 
		\int_{(0,\infty)} \frac{1}{2x} ( 1 - e^{-2 x t} )^3 \pi(\dd x) 
		\leq t^{1+\varepsilon} 2^{\varepsilon} \int_{(0,\infty)} x^\varepsilon \pi(\dd x).
	\end{split}
	\]
	Therefore, \eqref{eq:f-int-ass2_mma} holds, and Theorem \ref{thm:cadlag} applies.
\end{proof}

Observe that \eqref{eq:supOU-int-ass}, along with the condition 
$\int_{(0,\infty)} x^{-1} \pi(\dd x) < \infty$, implies that $\pi$ is a finite measure. In contrast,  Proposition 1 in \cite{GK} shows that if 
$\pi((0,\infty)) = \infty$ and the jump measure $\lambda$ has unbounded 
support, then the supremum of the supOU process is unbounded 
on any nonempty interval. In particular, the process 
cannot have a c\`adl\`ag modification. Therefore, the conditions in 
Corollary \ref{corr:supOU} are close to optimal.

\subsection{Well-balanced  supOU}

Inspired by the ideas in \cite{Barn, SchnurrWoerner2011}, 
in this example, we consider the kernel given by $f(x,s) = e^{-x|s|}$, with 
$V= ( 0,\infty)$ and $s \in \R$. When $x$ is a fixed parameter, i.e.,~$V$ 
consists of a single element, the moving average obtained by this function is known as the \textit{well-balanced} (L\'evy driven) OU process, and
it was introduced in \cite{SchnurrWoerner2011}. Motivated by this, in what follows, we will refer to the MMA associated to this kernel as a \textit{well-balanced  supOU}. As with the supOU process, a well-balanced supOU is well-defined if and only if \eqref{eq:supOU-cond} holds, 
which will be assumed throughout this example. 

Unlike the standard (L\'evy driven) OU process, the well-balanced OU is always continuous. It turns out that this property is preserved after superposition. 

\begin{corollary}\label{corr:wellbalsupOU}
	If the assumptions of Corollary \ref{corr:supOU} hold, then the well-balanced
	supOU process
	\[
	X(t) = \int_{(0,\infty) \times \R} e^{-x|t-s|} \Lambda(\dd x, \dd s)
	\]
	has a continuous modification.
\end{corollary}

\begin{proof}
	We show that the conditions of Theorem \ref{thm:cont} are satisfied. Conditions 
	\eqref{eq:f-ass-bound} and \eqref{eq:f-g-ass} are satisfied with $A=\emptyset$ and
	\[
	g(x,s) = 
	\begin{cases}
		e^{-xs}, & s \geq 0, \\
		1, & s \in [-1,0], \\
		e^{x(s+1)}, &  s < -1.
	\end{cases}
	\]
	Next, a straightforward calculation gives that 
	\[
	\int_{\R} \left( e^{-x|t-u|} - e^{-x|u|} \right)^2 \dd u 
	= \frac{2}{x} \left( 1- e^{-xt} ( 1 + xt) \right).
	\]
	Therefore, assuming \eqref{eq:supOU-int-ass},
	\[
	\begin{split}
		\dint_{V \times \R} (f(x,t-u) - f(x,-u))^2 \pi(\dd x) \dd u 
		\leq c t^{1 + \varepsilon} \int_{(0,\infty)} x^{\varepsilon} \pi(\dd x),
	\end{split}
	\]
	which shows that \eqref{eq:cont-ass}  holds with $\alpha = 2$.
\end{proof}

In view of
\[
X(t) = \dint_{(0,\infty) \times (-\infty, t]} e^{-x(t-s)} \Lambda(\dd x, \dd s)+\dint_{(0,\infty) \times (t, \infty)} e^{-x(s-t)} \Lambda(\dd x, \dd s),
\]
it follows by the same argument as for the supOU that if
$\pi((0,\infty)) = \infty$ and 
the jump measure $\lambda$ has unbounded support, then 
$X$ does not admit a c\`adl\`ag modification. Hence, the conditions in the preceding corollary are once again close to optimal.

Let us note that the classical well-balanced OU has not only continuous 
paths but, in fact, absolutely continuous paths, as shown in
\cite{SchnurrWoerner2011}. The sufficient conditions for absolute continuity of MMA processes are given in Theorem 3.1 in \cite{BasseRos2013b}. In our setting, these conditions reduce to the kernel function $f$ being absolutely continuous and its derivative $\dot{f}(x,s) = \frac{\partial}{\partial s} f(x,s)$ satisfying
\begin{equation}\label{eq:bor-cond}
	\trint_{V\times\R\times (0,\infty)} 
	\left( \vert z \dot{f}(x,s)\vert^2 \wedge \vert z \dot{f}(x,s)\vert \right) \pi(\dd x) \dd s \lambda(\dd z) <\infty.
\end{equation}
Note that this implies $\int_{(1,\infty)} z \lambda(\dd z)<\infty$. However, by decomposing $X=X_1+X_2$ as in \eqref{eq:X12-def}, we can replace \eqref{eq:bor-cond} with the following
\begin{align}
	\trint_{V\times\R\times (0,1]} 
	\left( \vert z \dot{f}(x,s)\vert^2 \wedge \vert z \dot{f}(x,s)\vert \right) \pi(\dd x) \dd s \lambda(\dd z) <\infty,\label{eq:ac-suff1}\\
	\trint_{V \times \R \times (1,\infty)} 1\wedge    \left(z\int_{a}^{b} \rvert\dot{f}(x,u-s) \rvert\dd u\right)\pi(\dd x) \dd s\lambda(\dd z)<\infty.\label{eq:ac-suff2}
\end{align}
Indeed, \eqref{eq:ac-suff1} is equivalent to \eqref{eq:bor-cond} for $X_1$. For $X_2$, we have by the deterministic Fubini theorem that
\begin{align*}
	X_2(b)-X_2(a) &= \int_a^b\trint_{V\times\R\times(1,\infty)} 
	z\dot{f}(x,u-s)  \mu(\dd x, \dd s, \dd z)\dd u \\
	&= \trint_{V\times\R\times(1,\infty)} \left( \int_a^b
	z\dot{f}(x,u-s)\dd u \right)  \mu(\dd x, \dd s, \dd z),
\end{align*}
whenever the latter integral exists a.s. The necessary and sufficient condition for this is \eqref{eq:ac-suff2}, implying the absolute continuity of $X_2$.

\begin{corollary} \label{corr:well-bal-supOU}
	Let $\lambda$ be a L\'evy measure on $(0,\infty)$, 
	and $\pi$ be a measure on $(0,\infty)$ such that \eqref{eq:supOU-cond} holds. If 
	\begin{equation} \label{eq:abs_con_welsupou}
		\dint_{(0,\infty) \times (0,1]} z ( 1 \wedge zx) \pi(\dd x) \lambda(\dd z) < \infty,
	\end{equation}
	then the well-balanced supOU process $X$ has absolutely continuous paths.
\end{corollary}

\begin{proof}
	Note that $f(x,\cdot)$ is absolutely continuous with derivative 
	\[
	\dot{f}(x,s)= - \mathrm{sgn}(s) xe^{-x\rvert s\rvert},\ s\neq0.
	\]
	A short calculation shows that condition \eqref{eq:ac-suff1} is equivalent to \eqref{eq:abs_con_welsupou}. To show \eqref{eq:ac-suff2}, we decompose the integral in \eqref{eq:ac-suff2} as $I_1+I_2+I_3$, where 
	\begin{align*}
		I_1= & \int_{(0,\infty)}\int_{(1,\infty)}
		\int_{-\infty}^{a} 1\wedge\left(z(e^{-x(a-s)}-e^{-x(b-s)})\right) 
		\dd s\lambda(\dd z)\pi(\dd x),\\
		I_2= & \int_{(0,\infty)}\int_{(1,\infty)}
		\int_{a}^{b}1\wedge\left(zx\int_{a}^{b}e^{-x\rvert u-s\rvert} 
		\dd u\right)\dd s\lambda(\dd z)\pi(\dd x),\\
		I_3=& \int_{(0,\infty)}\int_{(1,\infty)} 
		\int_{b}^{\infty}1\wedge\left(z(e^{-x(s-b)}-e^{-x(s-a)})\right) 
		\dd s\lambda(\dd z)\pi(\dd x).
	\end{align*}
	Clearly $I_{2}\leq(b-a)\pi((0,\infty))\lambda((1,\infty))<\infty$, and 
	\begin{align*}
		I_{1} &\leq \int_{0}^{\infty}\int_{1}^{\infty}\int_{0}^{\infty}1\wedge  ze^{-xs} \dd s\lambda(\dd z)\pi(\dd x)\\
		&=\left(\lambda((1,\infty))+\int_{(1,\infty)}\log(z)\lambda(\dd z)\right)\int_{(0,\infty)}\frac{\pi(\dd x)}{x}<\infty.
	\end{align*}
	The same computation shows that $I_3<\infty$. 
\end{proof}

Note that if $\int_{(0,1]} z \lambda(\dd z) < \infty$, then 
\eqref{eq:abs_con_welsupou} is satisfied whenever the process exists.
In particular, in this case, the process always has
absolutely continuous paths. On the other hand, if 
$\int_{(0,1]} z \lambda(\dd z) = \infty$, then the conditions 
of Corollary \ref{corr:wellbalsupOU} for continuity are, in general, weaker than \eqref{eq:abs_con_welsupou}. In fact, under some additional 
assumptions on the L\'evy measure $\lambda$, Theorem 3.3 in 
\cite{BasseRos2013b} states that \eqref{eq:bor-cond} is 
also necessary for absolutely continuous paths.
For example, if $\lambda(\dd z)=z^{-\alpha-1}\dd z$, 
for $1<\alpha<2$, then 
$\int_{(0,1]} z\lambda(\dd z)=\infty$ and 
\[
\int_{(u,\infty)} z\lambda(\dd z)= \frac{u^{1-\alpha}}{\alpha-1}\ 
\text{ and }\ 
\int_{(0,u]}z^{2}\lambda(\dd z)=\frac{u^{2-\alpha}}{2-\alpha}.
\]
Therefore, condition (3.7) of Theorem 3.3 in \cite{BasseRos2013b} holds.
Thus, $X$ has absolutely continuous paths if and only if  
\eqref{eq:bor-cond} holds. 
A short calculation gives that \eqref{eq:bor-cond} is equivalent to  
$\int_{(1, \infty)} x^{\alpha-1}\pi(\dd x) < \infty$. 
This reveals that there are measures $\pi$ such that 
the assumptions of Corollary \ref{corr:wellbalsupOU} hold, 
but \eqref{eq:bor-cond} do not. That is, unlike the classical 
well-balanced OU, $X$ can be continuous without being absolutely continuous.

\subsection{Trawl processes}

Let $V = [0,\infty)$, $\pi$ be the Lebesgue measure and 
$f(x,s) = \bone(0 \leq x \leq a(s))$, $s \geq 0$,
for a \textit{trawl function} 
$a \colon [0,\infty) \to [0,\infty)$, where 
$\int_0^\infty a(s) \dd s < \infty$. 
The \textit{trawl process} is defined as
\begin{equation} \label{eq:trawl-def}
	X(t) = \dint_{[0,\infty) \times (-\infty, t]} \bone ( 0 \leq x \leq a(t-s) )
	\Lambda(\dd x, \dd s).
\end{equation}

Trawl processes were introduced by Barndorff-Nielsen \cite{barndorff2011statID}.
Trawls have similar behavior to supOU processes in various aspects.
The marginals and dependence structure can be modeled independently,
and the two classes of processes obey similar type limit theorems, 
see \cite{pakkanen2021,talarczyk2020}.
For properties and applications of trawl processes, 
we refer to Chapter 8 in the monograph \cite{barndorff2018ambit}.
We are not aware of any results on the existence of c\`adl\`ag
modifications of trawl processes.

A function $a$ is $\delta$-H\"older continuous in a neighborhood of 0
for $\delta > 0$, if there exist $C > 0$ and $t_0 > 0$, such that 
for any $s,t \in [0,t_0]$
\[
|a(t) - a(s)| \leq C |t-s|^\delta.
\]
As a corollary of Theorem \ref{thm:cadlag}, we obtain the following.

\begin{corollary}
	Let $X$ be a trawl process in \eqref{eq:trawl-def} with non-increasing trawl function $a$ which is $\delta$-H\"older continuous in a neighborhood of 0
	for some $\delta > 0$. Then $X$ has a c\`adl\`ag modification.
\end{corollary}

\begin{proof}
	Since $a$ is non-increasing, \eqref{eq:f-ass-bound} and \eqref{eq:f-g-ass} hold with $A = \emptyset$ and
	\[
	g(x,s) =
	\begin{cases}
		f(x,s), & s \geq 0, \\
		f(x,0), & s \in [-1, 0], \\
		0, & s < - 1.
	\end{cases}
	\]
	Indeed, \eqref{eq:f-ass-bound} is easy to check, and \eqref{eq:f-g-ass} holds
	as
	\[
	\trint_{(0,\infty) \times \R \times (1,\infty)} 
	(g(x,s) z \wedge 1) \dd x \dd s \lambda(\dd z)
	= a(0) \overline \lambda(1) + \int_0^\infty a(s) \dd s < \infty.
	\]

	Next, we prove that
	\eqref{eq:f-int-ass1_mma} and \eqref{eq:f-int-ass2_mma} hold with 
	$\alpha = 2$.
	For $u \leq s \leq t$, we have
	\begin{equation*}
		f(x,t-u) -f(x,s-u) = - \bone ( a(t-u) < x \leq a(s-u) ).
	\end{equation*}
	Therefore,
	\[
	\begin{split}
		& \dint_{V \times \R} (f(x,t-u) - f(x,-u))^2 \dd x \dd u \\
		& = 
		\int_0^\infty (a(u) - a(t+u)) \dd u + \int_0^t a(u) \dd u \\
		& = 2 \int_0^t a(u) \dd u \leq 2 a(0) t,
	\end{split}
	\]
	showing that \eqref{eq:f-int-ass1_mma} holds.

	Next, we show \eqref{eq:f-int-ass2_mma}. Note that for $u > s$, the integrand is 0. For $u \in (0,s]$
	\[
	\begin{split}
		& |f(x,t-u) - f(x,s-u) | \wedge |f(x,s-u) - f(x,-u)| \\
		& = 	\bone( \{ (x,u): a(t-u) < x \leq a(s-u) \}) \wedge 
		\bone( \{ (x,u):  x \leq a(s-u) \}) \\
		& = \bone( \{ (x,u): a(t-u) < x \leq a(s-u) \}),
	\end{split}
	\]
	while for $u \leq 0$
	\[
	\begin{split}
		& |f(x,t-u) - f(x,s-u) | \wedge |f(x,s-u) - f(x,-u)| \\
		& = 
		\bone( \{ (x,u): a(t-u) < x \leq a(s-u) \}) \\
		& \quad \wedge 
		\bone( \{ (x,u): a(s-u) <  x \leq a(-u) \}) = 0.
	\end{split}
	\]
	Thus,
	\[
	\begin{split}
		& \dint_{V \times \R} \left( (f(x,t-u) - f(x,s-u))^2 \wedge 
		(f(x,s-u) - f(x,-u))^2 \right) \dd x \dd u \\
		& = \int_0^s (a(s-u) - a(t-u) ) \dd u  .
	\end{split}
	\]
	
	The proof is now completed by noticing that, on the one hand, for $t \leq t_0$, the H\"older continuity implies that 
	\[ \int_0^s (a(s-u) - a(t-u) ) \dd u   \leq C  \int_0^s (t-s)^\delta \dd u \leq C t^{1 + \delta},\]
	while, on the other hand, if $t>t_0$, we have
	\[ 
	\int_0^s (a(s-u) - a(t-u) ) \dd u   \leq \frac{a(0)}{t_0^\delta}t^{1 + \delta},
	\]
	due to the monotonicity of $a$.
\end{proof}

Next, we obtain some conditions which guarantee 
that a c\`adl\`ag modification cannot exist.
Let $A = \{ ( x,s) : s \leq 0,  0 \leq x \leq a(-s)\}$ and 
$A_t = (0,t) + A$. For $h > 0$ put 
\[
\bar{A}^{h} = \cup_{t \in [0,h]} A_t.
\]
First, we need a simple lemma. Write $|\cdot|$ for the two-dimensional 
Lebesgue measure.

\begin{lemma}
	If $|\bar{A}^{h_0}| = \infty$ for some $h_0 > 0$, then $|\bar{A}^{h}| = \infty$ for 
	all $h > 0$.
\end{lemma}

\begin{proof}
	Clearly, $\bar{A}^{h_1} \subset \bar{A}^{h_2}$ for $h_1 < h_2$, and
	$\bar{A}^{h} = \bar{A}^{h/2} \cup (\bar{A}^{h/2} + (0,h/2))$. Since the Lebesgue measure is 
	translation invariant, if $|\bar{A}^{h}| = \infty$, then $|\bar{A}^{h/2} | = \infty$, proving 
	the statement.
\end{proof}

Then we have a condition for non-c\`adl\`ag paths.

\begin{lemma}
	Assume that $\lambda$ has unbounded support, and $A$ is such that 
	$|\bar{A}^{h}| = \infty$ for all $h > 0$. Then for any $h > 0$ 
	\[
	\sup_{t \in [0,h]} X(t) = \infty \quad a.s.
	\]
	In particular, $X$ does not have a c\`adl\`ag modification. 
\end{lemma}

\begin{proof}
	Fix $x > 0$. Since $\nu ( \bar{A}^{h} \times [x,\infty)) = \infty$,
	with probability 1, there is a point $(\tau, \xi, \zeta) \in \bar{A}^{h} \times [x, \infty)$ of the Poisson random measure $\mu$ from \eqref{eq:Lambda}. This means that for some $t \in [0,h]$, $(\tau, \xi) \in A_t$, thus 
	$ X(t) = \Lambda(A_t) \geq x $.
\end{proof}

In case of monotone trawls, if $a$ is not bounded and the jump measure 
$\lambda$ has unbounded support, then there is no c\`adl\`ag modification.

\subsection{Power-weighted supOU processes}

Let $V=(0,\infty)$ and $f(x,s)= x^{\kappa} e^{-xs}\bone(s \geq 0)$ for some $\kappa>0$. We consider the \textit{power-weighted supOU process} defined as
\begin{equation}\label{eq:new:def}
	X(t) = \dint_{(0,\infty) \times (-\infty, t]} x^{\kappa} e^{-x(t-s)} \Lambda(\dd x, \dd s).
\end{equation}

First, we obtain conditions for the existence of such a process. We start with the case when $\int_{(0,1]} z \lambda(\dd z) < \infty$, so that no compensation is needed.

\begin{lemma} \label{lemma:wsupOU-fv-int}
	The Poisson integral 
	\[
	\trint_{(0,\infty)\times (-\infty, t] \times (0,\infty)}
	x^\kappa e^{-x(t-s)} z \mu(\dd x, \dd s, \dd z)
	\]
	exists if and only if 
	\begin{equation} \label{eq:wsupOU-iff}
	\begin{split}
		\dint_{(0,\infty)^2} &
		\left[ 
		x^{\kappa - 1} z \bone (x^\kappa z \leq 1)\right. \\
		& \quad \left. + ( 1 + \log (x^\kappa z) ) x^{-1} \bone(x^\kappa z > 1) \right] 
		\lambda(\dd z ) \pi (\dd x) < \infty.
	\end{split}
	\end{equation}
\end{lemma}

\begin{proof}
	By \cite[Theorem 2.7]{Kyprianou}, the integral exists if and only if
	\[
	\trint_{(0,\infty)^2 \times (0,\infty)} (x^\kappa e^{-xs} z \wedge 1)
	\nu(\dd x, \dd s, \dd z) < \infty,
	\]
	which is equivalent to \eqref{eq:wsupOU-iff}.
\end{proof}

Sufficient conditions for \eqref{eq:wsupOU-iff} can be obtained 
from the behavior of $\lambda$ at 0 and at $\infty$.
Let 
\begin{equation} \label{eq:wsupOU-beta}
	\beta_0 = \inf \left\{
	\beta' \geq 0 \colon \int_{(0,1]} z^{\beta'} \lambda (\dd z) < \infty
	\right\} \in [0,2],
\end{equation}
denote the Blumenthal--Getoor index of the L\'evy measure $\lambda$, and 
let 
\begin{equation*}
	\eta_\infty = \sup \left\{ \eta' \geq 0 \colon \int_{(1,\infty)} z^{\eta'} \lambda(\dd z)
	< \infty \right\} \in [0,\infty].
\end{equation*}
In what follows, if $\int_{(0,1]} z^{\beta_0} \lambda (\dd z) < \infty$, we set 
$\beta = \beta_0$. Otherwise, we choose $\beta > \beta_0$ arbitrarily close.
Similarly, if $\eta_\infty<\infty$, we put $\eta = \eta_\infty$ whenever  
$\int_{(1,\infty)} z^{\eta_\infty} \lambda(\dd z ) < \infty$, otherwise we take
$\eta < \eta_\infty$ arbitrarily close. If $\eta_\infty=\infty$, we set $\eta=\infty$.

\begin{lemma} \label{lemma:wsupOU-fv-suff}
	Assume that $\int_{(0,1]} z \lambda(\dd z) < \infty$, 
	$\int_{(1,\infty)} \log z  \lambda(\dd z) < \infty$,
	and  that
	\begin{equation} \label{eq:wsupOU-fv-intcond}
		\begin{split}
			& \int_{(0,1]} x^{(\eta \wedge 1) \kappa - 1} 
			\left( 1 + \bone ( \eta > 0) \log x^{-1} \right) \pi(\dd x)  < \infty 
			\quad \text{and } \\ 
			& \int_{(1,\infty)} x^{\beta \kappa -1} 
			\left( 1 + \bone (\beta = 0) \log x \right) \pi (\dd x) < \infty.
		\end{split}
	\end{equation}
	Then \eqref{eq:wsupOU-iff} holds.
\end{lemma}

\begin{proof}
	For the first term in \eqref{eq:wsupOU-iff}, for $x \in (1,\infty)$ we have
	\[
	\begin{split}
		& \int_{(1,\infty)} \int_{(0,\infty)} x^{\kappa -1} z 
		\bone(x^\kappa z \leq 1) \lambda(\dd z) \pi(\dd x) \\
		& \leq 
		\int_{(1,\infty)}
		\int_{(0,1]} z^\beta \lambda(\dd z) x^{\kappa - 1 - \kappa(1-\beta)}  \pi(\dd x) \\
		& = \int_{(0,1]} z^\beta \lambda(\dd z) \int_{(1,\infty)} x^{\beta \kappa -1}
		\pi(\dd x) < \infty.
	\end{split}
	\]
	Consider now $x \in (0,1]$. If $\eta < 1$, then
	\begin{equation} \label{eq:wsupOU-fv-1}
		\begin{split}
			& \int_{(0,1]} \int_{(0,\infty)} x^{\kappa -1} z 
			\bone(x^\kappa z \leq 1) \lambda(\dd z) \pi(\dd x) \\
			& = \int_{(0,1]}
			\left( \int_{(0,1]} z \lambda(\dd z) 
			+ \int_{(1,\infty)} z^\eta x^{-\kappa(1-\eta)} \lambda(\dd z) \right) 
			x^{\kappa - 1}  \pi(\dd x) \\
			& \leq 
			\int_{(0,1]} x^{\kappa -1} 
			\pi(\dd x) \int_{(0,1]} z \lambda(\dd z) \\
			& \qquad + 
			\int_{(0,1]} x^{\eta \kappa - 1} \pi(\dd x) 
			\int_{(1,\infty)} z^\eta \lambda(\dd z)  < \infty,
		\end{split}
	\end{equation}
	and similarly for $\eta \geq 1$
	\begin{equation} \label{eq:wsupOU-fv-2}
		\begin{split}
			& \int_{(0,1]} \int_{(0,\infty)} x^{\kappa -1} z 
			\bone(x^\kappa z \leq 1) \lambda(\dd z)  \pi(\dd x) \\
			& \leq \int_{(0,1]} x^{\kappa - 1} \pi(\dd x) 
			\int_{(0,\infty)} z \lambda(\dd z).
		\end{split}
	\end{equation}
	For the second term in \eqref{eq:wsupOU-iff}, for $x \in (0,1]$,
	we obtain
	\begin{equation} \label{eq:wsupOU-fv-3}
		\begin{split}
			& \int_{(0,1]} \int_{(1,\infty)} 
			( 1 + \log (x^\kappa z) ) x^{-1} \bone ( x^\kappa z > 1) \lambda(\dd z) \pi(\dd x) \\
			& \leq \int_{(0,1]} \int_{(1,\infty)} x^{-1} ( 1 + \log z)
			\bone ( x^\kappa z > 1) \lambda(\dd z) \pi(\dd x) \\
			& \leq 
			\begin{cases}
				c \int_{(0,1]} 
				x^{\eta \kappa - 1} \kappa \log x^{-1} 
				\pi(\dd x) \int_{(1,\infty)} z^\eta \lambda(\dd z) < \infty, & \eta > 0, \\
				\int_{(0,1]} x^{-1} \pi(\dd x) 
				\int_{(1,\infty)} ( 1 + \log z) \lambda(\dd z) < \infty, & \eta = 0,
			\end{cases}
		\end{split}
	\end{equation}
	and for $x \in (1,\infty)$
	\begin{equation} \label{eq:wsupOU-fv-4}
		\begin{split}
			& \int_{(1,\infty)} \int_{(0,\infty)} 
			\left( 1 + \log (x^\kappa z) \right) x^{-1} \bone(x^\kappa z > 1 ) \lambda(\dd z) \pi(\dd x) \\
			& \leq \int_{(1,\infty)} x^{-1} \left( 
			c \int_{(0,1]} (x^\kappa z)^\beta \lambda(\dd z) \right. \\
			& \hspace{2.3cm} \left.
			+ \left[ \overline \lambda(1) (1 + \kappa \log x)  + 
			\int_{(1,\infty)} \log z \lambda(\dd z) \right] \right) \pi(\dd x)  < \infty, 
		\end{split}
	\end{equation}
	for $\beta > 0$, while for $\beta = 0$, the first term in the 
	bracket above changes to $(1 + \log x) \lambda((0,1])$.
\end{proof}

Next, we turn to the compensated integral and obtain sufficient
conditions for its existence.

\begin{lemma} \label{lemma:wsupOU-comp}
	The compensated integral
	\[
	\trint_{(0,\infty) \times (-\infty,t ] \times (0,1]} 
	x^{\kappa} e^{-x(t-s)} z (\mu - \nu)(\dd x, \dd s, \dd z)
	\]
	exists if 
	\[
	\int_{(0,1]} x^{2\kappa -1} \pi(\dd x) < \infty \quad 
	\text{and } \
	\int_{(1,\infty)} x^{(\beta \vee 1)\kappa-1} \pi(\dd x)<\infty.
	\]
\end{lemma}

\begin{proof}
	We check the conditions of \cite[Proposition 34]{barndorff2018ambit}
	(see also \cite[Theorem 2.7]{rajputrosinski1989}), which in our setup
	read as
	\begin{align}
		& \dint_{(0,\infty)\times \R} \left| 
		\int_{(0,1]} \left( \tau(z f(x, s)) - f(x, s) \tau(z) \right) \lambda(\dd z) \right| \pi(\dd x) \dd s <\infty,\label{eq:rajros1}\\
		& \trint_{(0,\infty)\times \R \times (0,1]}  
		\left( 1 \wedge |z f(x, s)|^2 \right)  \pi(\dd x) \dd s \lambda(\dd z)< \infty,\label{eq:rajros2}
	\end{align}
	where $\tau(y) = y \bone_{[0,1]}(| y |)$. 
	
	For \eqref{eq:rajros1}, we have
	\begin{equation} \label{eq:wsupOU-aux1}
		\begin{split}
			&\trint_{(0,\infty)^2 \times (0,1]} \left| \tau(z f(x, s)) - f(x, s) \tau(z) \right|  \pi(\dd x) \dd s \lambda(\dd z)\\
			&= \trint_{(0,\infty)^2 \times (0,1]} z f(x, s) \bone(z f(x, s)>1) 
			\pi(\dd x)  \dd s \lambda(\dd z) \\
			&= \int_{(0,\infty)} \int_{(0,1]} 
			\int_0^{x^{-1} \log (z x^{\kappa})}
			z x^{\kappa} e^{-xs} \bone(z > x^{-\kappa}) \dd s \lambda(\dd z) \pi(\dd x)\\
			&=\int_{(0,\infty)} \int_{(0,1]} x^{-1} (1- z^{-1} x^{-\kappa}) z x^{\kappa} \bone(z > x^{-\kappa}) \lambda(\dd z) \pi(\dd x)\\
			&\leq  \int_{(1,\infty)} \int_{(0,1]}  x^{\kappa-1} z \bone(z > x^{-\kappa}) \lambda(\dd z) \pi(\dd x) \\
			& \leq 
			\begin{cases}
				\int_{(1,\infty)} x^{\kappa-1} \pi (\dd x) \int_{(0,1]}  z \lambda(\dd z), & \beta \leq 1,\\
				\int_{(1,\infty)} x^{\beta \kappa-1} \pi (\dd x) \int_{(0,1]}  z^\beta \lambda(\dd z), & \beta > 1.
			\end{cases}
		\end{split}
	\end{equation}
	
	We now turn to \eqref{eq:rajros2} and split the integral as follows
	\begin{align*}
		&\trint_{(0,\infty)^2\times (0,1]} 
		\left( 1 \wedge |z f(x, s)|^2 \right) \pi(\dd x) \dd s \lambda(\dd z) \\
		&=\trint_{(0,\infty)^2\times (0,1]} 
		\left[ 
		\bone(z x^\kappa e^{-x s} \geq 1)  \right. \\
		& \hspace{3.2cm} \left. 
		+ z^2 x^{2\kappa} e^{-2x s} \bone(z x^\kappa e^{-x s} < 1) 
		\right] \pi(\dd x) \dd s \lambda(\dd z) \\
		& =: J_1 + J_2.
	\end{align*}
	For $J_1$, we get
	\begin{align*}
		J_1 &= \int_{(1,\infty)}\int_{(0,1]}  x^{-1} \log(z x^\kappa) 
		\bone(z> x^{-\kappa}) \lambda(\dd z) \pi(\dd x) \\
		& \leq  \int_{(1,\infty)} \int_{(0,1]} z x^{\kappa-1} 
		\bone(z> x^{-\kappa}) \lambda(\dd z) \pi(\dd x) < \infty,
	\end{align*}
	where we used the simple bound $\log y \leq y$ for $y>1$ and 
	the last inequality in \eqref{eq:wsupOU-aux1}.
	For $J_2$, it follows that
	\begin{align*}
		J_2 &=  \int_{(0,\infty) \times (0,1]} 
		z^2 x^{2\kappa}  \bone(z \leq x^{-\kappa}) 
		\int_0^\infty e^{-2x s} \dd s 
		\pi(\dd x)  \lambda(\dd z)\\
		&\quad + \dint_{(0,\infty)\times (0,1]} 
		z^2 x^{2\kappa}  \bone(z > x^{-\kappa})
		\int_{x^{-1} \log (z x^{\kappa})}^\infty
		e^{-2x s}
		\dd s \pi(\dd x) \lambda(\dd z) \\
		&=  \frac{1}{2} 
		\dint_{(0,\infty) \times (0,1]} 
		\left( x^{2\kappa -1} z^2 \bone(z \leq x^{-\kappa}) +
		x^{-1} \bone(z > x^{-\kappa}) \right) \lambda(\dd z) \pi(\dd x).
	\end{align*}
	For the first term above, we have
	\begin{align*}
		&\int_{(0,\infty)} \int_{(0,1]} x^{2\kappa -1} z^2 \bone(z \leq x^{-\kappa}) \lambda(\dd z) \pi(\dd x)\\
		& = \int_{(0,1]} x^{2\kappa -1} \pi(\dd x) \int_{(0,1]} z^2 \lambda(\dd z) \\
		& \quad + \int_{(1,\infty)} x^{2\kappa -1 - \kappa (2-\beta)} \pi(\dd x) 
		\int_{(0,1]} z^\beta \lambda(\dd z) < \infty,
	\end{align*}
	while for the second
	\[
	\begin{split}
	& \int_{(1,\infty)} \int_{(0,1]} x^{-1} \bone(z > x^{-\kappa}) \lambda(\dd z) \pi(\dd x) \\
	& \leq \int_{(1,\infty)} x^{\beta \kappa -1}  \pi(\dd x) \int_{(0,1]} z^\beta \lambda(\dd z) < \infty.
	\end{split}
	\]
\end{proof}

Note also that
\begin{equation} \label{eq:wsupOU-det}
	\dint_{(0,\infty)\times \R} f(x, s) \pi(\dd x) \dd s = \int_{(0,\infty)} x^{\kappa-1} \pi(\dd x).
\end{equation}

Combining Lemmas \ref{lemma:wsupOU-fv-int}, \ref{lemma:wsupOU-fv-suff},
\ref{lemma:wsupOU-comp}, and \eqref{eq:wsupOU-det}, we can give
sufficient conditions for the existence of the 
power-weighted supOU process in \eqref{eq:new:def}. For instance, a simple sufficient condition for existence is that $\int_{(1,\infty)} \log z \lambda(\dd z)<\infty$, $\int_{(0,1]} x^{-1} \pi (\dd x) <\infty$, and $\int_{(1,\infty)} x^{(\beta \vee 1)\kappa-1} \log x \pi(\dd x)<\infty$.

\begin{corollary} \label{corr:wsupOU}
	Assume that the power-weighted supOU process in \eqref{eq:new:def} is well defined, 
	$\int_{(1,\infty)} \log z \lambda(\dd z) < \infty$ and 
	\eqref{eq:wsupOU-fv-intcond} hold.
	If for some $\varepsilon > 0$
	\begin{equation} \label{eq:wsupOU-int-ass}
		\int_{(1,\infty)} x^{\kappa + \varepsilon} \pi(\dd x) < \infty,
	\end{equation}
	then $X$ has a c\`adl\`ag modification.
\end{corollary}

\begin{proof}
	For the conditions \eqref{eq:f-ass-bound} and \eqref{eq:f-g-ass}, take $A = \emptyset$ and
	\[
	g(x,s) =
	\begin{cases}
		x^\kappa e^{-xs}, & s \geq 0, \\
		x^\kappa, & s \in [-1,0], \\
		0, & s < -1.
	\end{cases}
	\]
	Clearly, \eqref{eq:f-ass-bound} holds and
	\[
	\begin{split}
		& \trint_{(0,\infty)^2 \times (1,\infty)} 
		( g(x,s) z \wedge 1) \pi(\dd x) \dd s \lambda(\dd z) \\
		& \leq \pi((0,\infty)) \overline \lambda(1) +
		\int_{(0,1)} \int_{(1,\infty)} x^{\kappa-1} z \bone(z \leq x^{-\kappa}) \lambda(\dd z) \pi(\dd x)\\
		&\quad + \int_{(0,\infty)} \int_{(1,\infty)} x^{-1} 
		( 1 + \log(z x^\kappa) ) \bone(z> x^{-\kappa}) \lambda(\dd z) \pi(\dd x).
	\end{split}
	\]
	The second term is finite by \eqref{eq:wsupOU-fv-1} ($\eta < 1$) and 
	\eqref{eq:wsupOU-fv-2} ($\eta \geq 1$), and the third term is finite by
	\eqref{eq:wsupOU-fv-3} and \eqref{eq:wsupOU-fv-4}.
	
	We now show that \eqref{eq:f-int-ass1_mma} and \eqref{eq:f-int-ass2_mma}
	are satisfied with $\alpha = 1 + \varepsilon'$, where $\varepsilon' = \varepsilon /(\kappa +1)$. 
	For $u \leq s \leq t$
	\[
	f(x,t-u) -f(x,s-u) = x^{\kappa} e^{x(u-s)} ( e^{-x(t-s)} -1).
	\]
	Thus, \eqref{eq:f-int-ass1_mma} holds since
	\begin{equation*}
		\begin{split}
			&\dint_{(0,\infty) \times (-\infty, 0]} 
			|f(x,t-u) -f(x,-u)|^\alpha \pi(\dd x) \dd u\\
			&= \int_{(0,\infty)} x^{\alpha \kappa} \frac{1}{\alpha x} ( 1 - e^{-x t} )^\alpha \pi(\dd x) \\
			& \leq t^{1+\varepsilon'} \frac{1}{\alpha} \int_{(0,\infty)} x^{\alpha \kappa + \varepsilon'} \pi(\dd x),
		\end{split}
	\end{equation*}
	and 
	\[
	\dint_{(0,\infty) \times (0,t]} f(x,t-u)^\alpha \pi(\dd x) \dd u 
	\leq t \int_{(0,\infty)} x^{\alpha \kappa} \pi(\dd x).
	\]
	To show \eqref{eq:f-int-ass2_mma}, note that the previous calculation shows that 
	\[
	\begin{split}
		& \dint_{V \times (-\infty, 0]} 
		| f(x,t-u) - f(x, s-u)|^\alpha \wedge
		| f(x,s-u) - f(x, -u)|^\alpha  \pi(\dd x) \dd u \\
		& \leq c t^{1+\varepsilon'}. 
	\end{split}
	\]
	Since $f(x,s-u) - f(x,-u) = 0$ for $u > s$, it remains to consider the integral over $(0,s]$. We similarly get
	\[
	\begin{split}
		& \dint_{(0,\infty) \times (0,s]} 
		|f(x,t-u) -f(x,s-u)|^\alpha \pi(\dd x) \dd u \\
		& = \int_{(0,\infty)} x^{\alpha \kappa} (1-e^{-x(t-s)})^\alpha   \int_0^s e^{-\alpha xu} 
		\dd u \pi(\dd x) \\
		& \leq 
		\int_{(0,\infty)}  x^{\alpha \kappa} \frac{1}{\alpha x} ( 1 - e^{-2 x t} )^{1 + \alpha} \pi(\dd x) 
		\leq t^{1+\varepsilon'} 2^{\varepsilon'} \int_{(0,\infty)} x^{\alpha \kappa+\varepsilon'} \pi(\dd x).
	\end{split}
	\]
	Therefore, \eqref{eq:f-int-ass2_mma} holds, and the claim now follows from Theorem \ref{thm:cadlag}.
\end{proof}

Next, we obtain conditions for the non-existence of a c\`adl\`ag
modification.

\begin{lemma} \label{lemma:wsupOU-noncadlag}
	Let 
	\[
	X(t) = \trint_{(0,\infty)\times (-\infty, t] \times (0,\infty)} x^\kappa e^{-x(t-s)} z \mu(\dd x, \dd s, \dd z),
	\]
	and assume that the conditions of Lemma \ref{lemma:wsupOU-fv-suff} hold. 
	If for each $r > 0$ 
	\begin{equation} \label{eq:wsupOU-non-cond}
		( \pi \times \lambda) ( \{ ( x,z): x^\kappa z > r \} ) = \infty,
	\end{equation}
	then for any $h > 0$ almost surely
	\[
	\sup_{t \in [0,h]} X(t) = \infty.
	\]
	In particular, $X$ does not have a c\`adl\`ag modification.
\end{lemma}

\begin{proof}
	It is enough to prove for $h = 1$. We have
	\begin{equation} \label{eq:wsupOU-sup}
		\sup_{t \in (0,1)} X(t) = \sup_{t \in (0,1)} \trint x^\kappa z e^{-x(t-s)} 
		\mu(\dd x, \dd s, \dd z) \geq 
		\sup_{\tau_i \in (0,1)} \xi_i^\kappa \zeta_i.
	\end{equation}
	By \eqref{eq:wsupOU-non-cond}, the supremum above is almost surely infinite.
\end{proof}

Assume that 
$\overline \lambda(z) = z^{-\beta_0} h(z)$ for $z \in (0,1]$, where 
$\beta_0$ is the Blumenthal--Getoor index, given in \eqref{eq:wsupOU-beta},
and $h$ is a function such that $\lim_{z \downarrow 0} h(z) z^\varepsilon = 0$
$\lim_{z \downarrow 0} h(z) z^{-\varepsilon} = \infty$ for any $\varepsilon > 0$.
This, in particular, includes regularly varying functions $\overline{\lambda}$. 

If for some $\varepsilon > 0$
\[
\int_{(1,\infty)} x^{\kappa \beta_0 - \varepsilon} \pi(\dd x) = \infty,
\]
then \eqref{eq:wsupOU-non-cond} holds for all $r > 0$. 
Therefore, we see that condition \eqref{eq:wsupOU-int-ass} is close to optimal
if $\beta_0 = 1$. Also note that the assumption 
$\int_{(0,1]} z \lambda(\dd z) < \infty$ implies that $\beta_0 \leq 1$.

\section{Proofs of Theorems \ref{thm:cadlag} and \ref{thm:cont}} 
\label{sect:proof}

Decompose $\Lambda$ as $\Lambda = \Lambda_1 + \Lambda_2$ with
\begin{align*}
	\Lambda_1(\dd x, \dd u) &= \int_{(0,1]} z (\mu - \nu)(\dd x, \dd u , \dd z),\\
	\Lambda_2(\dd x, \dd u) &= 
	m \pi(\dd x) \dd u +
	\int_{(1,\infty)} z \mu(\dd x, \dd u , \dd z).
\end{align*}
Without loss of generality, we may and do assume that $m = 0$.
Let 
\begin{equation} \label{eq:X12-def}
	X_i (t) = \dint_{V \times \R} f(x,t-u) \Lambda_i(\dd x, \dd u), \quad 
	i = 1,2,
\end{equation}
denote the processes associated to $\Lambda_1$ and 
$\Lambda_2$, respectively.  The following result provides a simple condition for the existence of a c\`adl\`ag modification for ID processes driven by $\Lambda_2$.

\begin{lemma} \label{lemma:X2}
	Assume that $f(x,\cdot)$ is  c\`adl\`ag for each $x \in V$, there exist a Borel set $A \subset V \times \R$ and a function 
	$g: V \to [0,\infty)$, such that $(\pi \times \mathrm{Leb}) (A) < \infty$, and 
	conditions \eqref{eq:f-ass-bound} and \eqref{eq:f-g-ass} hold.
	Then the process
	\[
	X_2(t) = \dint_{V \times \R} f(x,t-s)  \Lambda_2(\dd x, \dd s), \,\,t\in[0,1],
	\]
	is well-defined and c\`adl\`ag. Moreover, if $f(x, \cdot)$ is continuous for each $x \in V$, then $X_2$ is continuous.
\end{lemma}

\begin{proof}
	Let $(\xi_k,\tau_k,\zeta_k)_{k\geq 0}$ be the points of the Poisson random measure $\mu_2$ corresponding to $\Lambda_2$ as in \eqref{eq:Lambda}. The process 
	$X_2$ can be written as
	\begin{equation*} 
		X_2(t) = \left( \dint_A + \dint_{A^c} \right) 
		f(x, t-s)	 \Lambda_2 (\dd x, \dd s) =: J_1 (t) + J_2(t).
	\end{equation*}
	Since $\nu(A \times (1,\infty) ) = (\pi \times \mathrm{Leb}) (A) \lambda((1,\infty)) 
	< \infty$, 
	\[
	J_1(t) = \sum_{(\xi_i, \tau_i) \in A} f(\xi_i, t- \tau_i) \zeta_i
	\]
	is almost surely a finite sum. Therefore, if $f(x, \cdot)$ is c\`adl\`ag for each $x \in V$, then $J_1$ is c\`adl\`ag as well. Similarly, $J_1$ is continuous whenever $f(x, \cdot)$ is continuous for each $x \in V$.
	By the assumption, for all $t \in [0,1]$
	\begin{equation*}
		\begin{split}
			\dint_{A^c}  f(x,t-s) \Lambda_2(\dd x, \dd s) & = \sum_{(\xi_i, \tau_i) \in A^c} \zeta_i  f(\xi_i, t-\tau_i) \\
			& \leq \sum_{(\xi_i, \tau_i) \in A^c} \zeta_i g(\xi_i, -\tau_i) \\
			& = \dint_{A^c} g(x,-u) \Lambda_2(\dd x, \dd u).
		\end{split}
	\end{equation*}
	The last integral is a.s.~finite by \eqref{eq:f-g-ass}. Thus, if $f(x, \cdot)$ is continuous for each $x \in V$, then
	the  Weierstrass M-test (e.g.~\cite[Theorem 7.10]{rudin1976}) implies that 
	a.s.~the series $\sum_{(\xi_i, \tau_i) \in A^c} \zeta_i  f(\xi_i,t- \tau_i)$ 
	converges uniformly in $t$ to a continuous limit. Furthermore, if $f(x, \cdot)$ is c\`adl\`ag for each $x \in V\times\R$, then the 
	same argument as in the Weierstrass M-test shows that the limit exists,
	and it is c\`adl\`ag, see \cite[Problem V.1]{Pollard}.
\end{proof}

The following lemma will be crucial in our analysis.

\begin{lemma} \label{lemma:decomp-ineq}
	For any $\alpha \in [1,2]$, there exists $C = C_\alpha > 0$ only depending on 
	$\alpha$, such that for any $y>0$ and $0 \leq s < t <\infty$
	\[
	\begin{split}
		& \p ( |X_1(t) - X_1(s) | \wedge |X_1(s) - X_1(0) | > y ) \\
		& \leq C y^{-2\alpha}
		\left( \int_{(0,1]} z^\alpha \lambda(\dd z) \right)^2
		\dint_{V \times \R} | f(x,t-u) - f(x, s-u)|^\alpha \pi(\dd x) \dd u \\
		& \quad \times 
		\dint_{V \times \R} | f(x,s-u) - f(x, -u)|^\alpha \pi(\dd x) \dd u \\
		& + C y^{-\alpha}
		\int_{(0,1]} z^\alpha \lambda(\dd z)
		\dint_{V \times \R} 
		\left( | f(x,t-u) - f(x, s-u)| \right. \\
		& \hspace{5cm} \left. \wedge | f(x,s-u) - f(x, -u)| \right)^\alpha 
		\pi(\dd x) \dd u.
	\end{split}
	\] 
\end{lemma}

\begin{proof}
	Fix $0 \leq s < t$ and let $B_1 = V \times (-\infty, 0]$, 
	$B_2 = V \times (0,s]$, and $B_3 = V \times (s,\infty)$.
	Decompose the increments as
	\begin{equation*} 
		\begin{split}
			& X_1(t) - X_1(s) \\ & = 
			\left( \dint_{B_1} + 
			\dint_{B_2} + \dint_{B_3} \right)
			\left( f(x,t-u) - f(x, s-u) \right) \Lambda_1(\dd x, \dd u) \\
			& =:  I_1(s,t) + I_2(s,t) + I_3(s,t),
		\end{split}
	\end{equation*}
	and similarly
	\begin{equation*} 
		\begin{split}
			& X_1(s) - X_1(0) \\
			& = 
			\left( \dint_{B_1} + 
			\dint_{B_2} + \dint_{B_3} \right)
			\left( f(x,s-u) - f(x,-u) \right) \Lambda_1(\dd x, \dd u)  \\
			& =:  I_1(0,s) + I_2(0,s) + I_3(0,s).
		\end{split}
	\end{equation*}
	
	Since the stochastic integrals are taken on disjoint sets, 
	$I_1, I_2$ and $I_3$ are independent.
	Therefore, 
	\begin{equation} \label{eq:dec-lemma-aux1}
		\begin{split}
			& \p ( |X_1(t) - X_1(s) | \wedge |X_1(s) - X_1(0) | > y ) \\
			& \leq \sum_{i\neq j} 
			\p ( |I_i(s,t) | > y/3 ) \p ( | I_j(0,s) | > y/3 ) \\
			& \quad  + \sum_{i=1}^3 \p ( |I_i(s,t) | > y/3, | I_i(0,s) | > y/3 ).
		\end{split}
	\end{equation}
	Introduce the notation
	\begin{equation*} 
		\begin{split}
			& A_{+}(s,t) = \{ (x,u): |f(x,t-u) - f(x,s-u) | 
			\geq |f(x,s-u) - f(x, -u) |  \}, \\
			& A_{-}(s,t) = \{ (x,u): |f(x,t-u) - f(x,s-u) | < |f(x,s-u) - f(x, -u) | \},
		\end{split}
	\end{equation*}
	and put for $i = 1,2,3$
	\[	
	\begin{split}
		I_{i,\pm}(s,t) &=  \dint_{B_i \cap A_{\pm}(s,t)}
		(f(x,t-u) - f(x,s-u)) \Lambda_1(\dd x , \dd u), \\
		I_{i,\pm}(0,s) &=  \dint_{B_i \cap A_{\pm}(s,t)}
		(f(x,s-u) - f(x,-u)) \Lambda_1(\dd x , \dd u).
	\end{split}
	\]
	Note that $I_{i,+}$ and $I_{i,-}$ are independent.
	Using also that
	\[
	\{ |I_i(0,s) | > {y}/{3} \} \subset 
	\{ |I_{i,+}(0,s) | > {y}/{6} \}
	\cup \{ |I_{i,-}(0,s) | > {y}/{6} \},
	\]
	we obtain 
	\begin{equation} \label{eq:dec-lemma-aux2}
		\begin{split}
			& \p ( |I_i(s,t) | > y/3, | I_i(0,s) | > y/3 ) \\
			& \leq 
			\p ( |I_{i,+}(0,s)| > y/6) + \p ( | I_{i, -}(s,t) | > y/6) \\
			& \quad  + \p ( |I_{i,+}(s,t)| > y/6)  \p ( | I_{i, -}(0,s) | > y/6).
		\end{split}
	\end{equation}
	According to Theorem 1 in \cite{MR},
	there exists $c_\alpha = c > 0$ such that for any measurable $h$
	\begin{equation} \label{eq:MR-ineq}
		\E \left| \dint_{V \times \R} h(x,u) \Lambda_1(\dd x, \dd u) \right|^\alpha 
		\leq c \int_{(0,1]} z^\alpha \lambda(\dd z) 
		\dint_{V \times \R} |h(x,u)|^\alpha \pi(\dd x) \dd u.
	\end{equation}
	The latter, combined with Markov's inequality, implies
	\begin{equation} \label{eq:dec-lemma-aux3}
		\begin{split}
			& \sum_{i=1}^3 \left( \p ( |I_{i,+}(0,s)| > y/6) + \p ( | I_{i, -}(s,t) | > y/6)
			\right) \\
			& \leq \frac{c 6^\alpha}{y^\alpha} \sum_{i=1}^3 
			\int_{(0,1]} z^\alpha \lambda(\dd z) 
			\left( 
			\dint_{B_i \cap A_{+}(s,t)} |f(x,s-u) - f(x,-u)|^\alpha \pi(\dd x) \dd u
			\right. \\
			& \quad \left. 
			+ \dint_{B_i \cap A_{-}(s,t)} |f(x,t-u) - f(x,s-u)|^\alpha \pi(\dd x) \dd u
			\right) \\
			& = \frac{c 6^\alpha \int_{(0,1]} z^\alpha \lambda(\dd z) }{y^\alpha} 
			\dint_{V \times \R} (|f(x,s-u) - f(x,-u)| \\
			& \hspace{4.5cm} \wedge |f(x,t-u) - f(x,s-u)|)^\alpha
			\pi(\dd x) \dd u.
		\end{split}
	\end{equation}
	Similarly,
	\begin{equation} \label{eq:dec-lemma-aux4}
		\begin{split}
			& \sum_{i\neq j} \p ( |I_i(s,t) | > y/3 ) \p ( | I_j(0,s) | > y/3 ) \\
			& \leq \left( \sum_{i=1}^3 \p ( |I_i(s,t) | > y/3) \right)
			\left( \sum_{i=1}^3 \p ( |I_i(0,s) | > y/3) \right) \\
			& \leq 
			\left( 
			\frac{c 3^{\alpha} \int_{(0,1]} z^\alpha \lambda(\dd z) }{y^\alpha}
			\right)^2
			\dint_{V \times \R} | f(x,t-u) - f(x, s-u)|^\alpha \pi(\dd x) \dd u \\
			&\quad \times 
			\dint_{V \times \R} | f(x,s-u) - f(x, -u)|^\alpha \pi(\dd x) \dd u,
		\end{split}
	\end{equation}
	and 
	\begin{equation} \label{eq:dec-lemma-aux5}
		\begin{split}
			& \sum_{i=1}^3 
			\p ( |I_{i,+}(s,t) | > y/6 ) \p ( | I_{i,-}(0,s) | > y/6 )  \\
			& \leq 
			\left( c \frac{6^{\alpha} \int_{(0,1]} z^\alpha \lambda(\dd z) }{y^\alpha}
			\right)^2
			\sum_{i=1}^3
			\dint_{B_i} | f(x,t-u) - f(x, s-u)|^\alpha \pi(\dd x) \dd u \\
			& \quad \times 
			\dint_{B_i} | f(x,s-u) - f(x, -u)|^\alpha \pi(\dd x) \dd u \\
			& \leq 
			\left( c \frac{6^{\alpha} \int_{(0,1]} z^\alpha \lambda(\dd z) }{y^\alpha}
			\right)^2
			\dint_{V \times \R} | f(x,t-u) - f(x, s-u)|^\alpha \pi(\dd x) \dd u \\
			& \quad \times 
			\dint_{V \times \R} | f(x,s-u) - f(x, -u)|^\alpha \pi(\dd x) \dd u.
		\end{split}
	\end{equation}
	Substituting back \eqref{eq:dec-lemma-aux2}, \eqref{eq:dec-lemma-aux3},
	\eqref{eq:dec-lemma-aux4}, and \eqref{eq:dec-lemma-aux5} into
	\eqref{eq:dec-lemma-aux1}, the result follows.
\end{proof}

Combining Lemmas \ref{lemma:X2} and \ref{lemma:decomp-ineq} with 
Billingsley's classical condition (\cite[Theorem 13.6]{Billingsley}),
we obtain the following.

\begin{proof}[Proof of Theorem \ref{thm:cadlag}]
	In view of Lemma \ref{lemma:X2}, we only need to show that $X_1$ has a c\`adl\`ag modification. We use \cite[Theorem 13.6]{Billingsley}. Specifically, we need to verify that $X_1$ is continuous in probability, and that there exist $\beta \geq 0$, $\varepsilon > 0$, $C > 0$, such that for $0 < s < t$ and $y > 0$
	\begin{equation*}
		\p ( |X_1(t) - X_1(s) | \wedge |X_1(s) - X_1(0) | > y ) \leq C y^{-\beta} t^{1 + \varepsilon},
	\end{equation*}
	due to the stationarity of $X_1$.	That the latter condition holds follows directly from Lemma \ref{lemma:decomp-ineq} and our assumptions. Therefore, the only thing needed to be proved is that $X_1$ is continuous in probability. This follows easily 
	from Markov's inequality combined with \eqref{eq:MR-ineq}. Indeed, for 
	$t > 0$ fixed, $0\leq s<t$, and $\varepsilon > 0$ arbitrary
	\begin{equation}\label{eq:mom-cond-cont}
		\begin{split}
			& \p ( |X_1(t) - X_1(s) | > \varepsilon ) \\
			& \leq \varepsilon^{-\alpha} \E |X_1(t) - X_1(s)|^\alpha \\
			& \leq \varepsilon^{-\alpha} c \int_{(0,1]} z^\alpha \lambda(\dd z)
			\dint_{V \times \R} 
			\left| f(x,t-u) - f(x,s-u) \right|^{\alpha} 
			\pi(\dd x) \dd u,
		\end{split}
	\end{equation}		
	where the upper bound tends to 0 as $s \uparrow t$ by \eqref{eq:f-int-ass1_mma}.
\end{proof}

\begin{proof}[Proof of Theorem \ref{thm:cont}]
	The proof follows from Lemma \ref{lemma:X2} and Kolmogorov's continuity theorem applied to $X_1$ using the moment bound \eqref{eq:mom-cond-cont}.
\end{proof}

\bigskip
\textbf{Acknowledgments}: DG was supported by the Croatian Science Foundation under the project Scaling in Stochastic Models (HRZZ-IP-2022-10-8081).


\begin{thebibliography}{10}
	
	\bibitem{Barn}
	O.~E. Barndorff-Nielsen.
	\newblock Superposition of {O}rnstein-{U}hlenbeck type processes.
	\newblock {\em Teor. Veroyatnost. i Primenen.}, 45(2):289--311, 2000.
	
	\bibitem{barndorff2011statID}
	O.~E. Barndorff-Nielsen.
	\newblock Stationary infinitely divisible processes.
	\newblock {\em Braz. J. Probab. Stat.}, 25(3):294--322, 2011.
	
	\bibitem{barndorff2018ambit}
	O.~E. Barndorff-Nielsen, F.~E. Benth, and A.~E.~D. Veraart.
	\newblock {\em Ambit Stochastics}, volume~88 of {\em Probability Theory and
		Stochastic Modelling}.
	\newblock Springer, Cham, 2018.
	
	\bibitem{bnleonenko2005}
	O.~E. Barndorff-Nielsen and N.~N. Leonenko.
	\newblock Spectral properties of superpositions of {O}rnstein-{U}hlenbeck type
	processes.
	\newblock {\em Methodol. Comput. Appl. Probab.}, 7(3):335--352, 2005.
	
	\bibitem{barndorff2014intTrawl}
	O.~E. Barndorff-Nielsen, A.~Lunde, N.~Shephard, and A.~E.~D. Veraart.
	\newblock Integer-valued trawl processes: a class of stationary infinitely
	divisible processes.
	\newblock {\em Scand. J. Stat.}, 41(3):693--724, 2014.
	
	\bibitem{barndorff2011multivariate}
	O.~E. Barndorff-Nielsen and R.~Stelzer.
	\newblock Multivariate sup{OU} processes.
	\newblock {\em Ann. Appl. Probab.}, 21(1):140--182, 2011.
	
	\bibitem{basse2020sufficient}
	A.~Basse-O'Connor.
	\newblock Sufficient conditions for c{\`a}dl{\`a}g sample paths of stable
	superpositions of {O}rnstein--{U}hlenbeck processes.
	\newblock In {\em 62nd ISI World Statistics Congress 2019}, pages 74--79.
	Department of Statistics Malaysia, 2020.
	
	\bibitem{BasseRos2013b}
	A.~Basse-O'Connor and J.~Rosi\'nski.
	\newblock Characterization of the finite variation property for a class of
	stationary increment infinitely divisible processes.
	\newblock {\em Stochastic Process. Appl.}, 123(6):1871--1890, 2013.
	
	\bibitem{BasseRos}
	A.~Basse-O'Connor and J.~Rosi\'nski.
	\newblock On the uniform convergence of random series in {S}korohod space and
	representations of c\`adl\`ag infinitely divisible processes.
	\newblock {\em Ann. Probab.}, 41(6):4317--4341, 2013.
	
	\bibitem{Billingsley}
	P.~Billingsley.
	\newblock {\em Convergence of probability measures}.
	\newblock Wiley Series in Probability and Statistics: Probability and
	Statistics. John Wiley \& Sons, Inc., New York, second edition, 1999.
	\newblock A Wiley-Interscience Publication.
	
	\bibitem{CDH}
	C.~Chong, R.~C. Dalang, and T.~Humeau.
	\newblock Path properties of the solution to the stochastic heat equation with
	{L}\'evy noise.
	\newblock {\em Stoch. Partial Differ. Equ. Anal. Comput.}, 7(1):123--168, 2019.
	
	\bibitem{curato2019}
	I.~V. Curato and R.~Stelzer.
	\newblock Weak dependence and {GMM} estimation of sup{OU} and mixed moving
	average processes.
	\newblock {\em Electron. J. Stat.}, 13(1):310--360, 2019.
	
	\bibitem{Fasen05}
	V.~Fasen.
	\newblock Extremes of regularly varying {L}\'{e}vy-driven mixed moving average
	processes.
	\newblock {\em Adv. in Appl. Probab.}, 37(4):993--1014, 2005.
	
	\bibitem{GK2}
	D.~Grahovac and P.~Kevei.
	\newblock Almost sure growth of integrated sup{OU} processes.
	\newblock {\em Bernoulli}, 31(4):2597--2623, 2025.
	
	\bibitem{GK}
	D.~Grahovac and P.~Kevei.
	\newblock Tail behavior and almost sure growth rate of superpositions of
	{O}rnstein-{U}hlenbeck-type processes.
	\newblock {\em J. Theoret. Probab.}, 38(1):Paper No. 1, 15, 2025.
	
	\bibitem{GLT19}
	D.~Grahovac, N.~N. Leonenko, and M.~S. Taqqu.
	\newblock Limit theorems, scaling of moments and intermittency for integrated
	finite variance sup{OU} processes.
	\newblock {\em Stochastic Process. Appl.}, 129(12):5113--5150, 2019.
	
	\bibitem{GLT2020multifaceted}
	D.~Grahovac, N.~N. Leonenko, and M.~S. Taqqu.
	\newblock The multifaceted behavior of integrated sup{OU} processes: the
	infinite variance case.
	\newblock {\em J. Theoret. Probab.}, 33(4):1801--1831, 2020.
	
	\bibitem{GLT21}
	D.~Grahovac, N.~N. Leonenko, and M.~S. Taqqu.
	\newblock Intermittency and infinite variance: the case of integrated sup{OU}
	processes.
	\newblock {\em Electron. J. Probab.}, 26:Paper No. 56, 31, 2021.
	
	\bibitem{Kyprianou}
	A.~E. Kyprianou.
	\newblock {\em Fluctuations of {L}\'{e}vy Processes with Applications}.
	\newblock Universitext. Springer, Heidelberg, second edition, 2014.
	\newblock Introductory lectures.
	
	\bibitem{Maejima}
	M.~Maejima.
	\newblock A self-similar process with nowhere bounded sample paths.
	\newblock {\em Z. Wahrsch. Verw. Gebiete}, 65(1):115--119, 1983.
	
	\bibitem{MR}
	C.~Marinelli and M.~R\"{o}ckner.
	\newblock On maximal inequalities for purely discontinuous martingales in
	infinite dimensions.
	\newblock In {\em S\'{e}minaire de {P}robabilit\'{e}s {XLVI}}, volume 2123 of
	{\em Lecture Notes in Math.}, pages 293--315. Springer, Cham, 2014.
	
	\bibitem{pakkanen2021}
	M.~S. Pakkanen, R.~Passeggeri, O.~Sauri, and A.~E.~D. Veraart.
	\newblock Limit theorems for trawl processes.
	\newblock {\em Electron. J. Probab.}, 26:Paper No. 116, 36, 2021.
	
	\bibitem{Pollard}
	D.~Pollard.
	\newblock {\em Convergence of stochastic processes}.
	\newblock Springer Series in Statistics. Springer-Verlag, New York, 1984.
	
	\bibitem{rajputrosinski1989}
	B.~S. Rajput and J.~Rosi\'{n}ski.
	\newblock Spectral representations of infinitely divisible processes.
	\newblock {\em Probab. Theory Related Fields}, 82(3):451--487, 1989.
	
	\bibitem{rm2022}
	A.~R{\o}nn-Nielsen and M.~Stehr.
	\newblock Extremes of {L}\'{e}vy-driven spatial random fields with regularly
	varying {L}\'{e}vy measure.
	\newblock {\em Stochastic Process. Appl.}, 150:19--49, 2022.
	
	\bibitem{rosinski1989}
	J.~Rosi\'{n}ski.
	\newblock On path properties of certain infinitely divisible processes.
	\newblock {\em Stochastic Process. Appl.}, 33(1):73--87, 1989.
	
	\bibitem{rudin1976}
	W.~Rudin.
	\newblock {\em Principles of mathematical analysis}.
	\newblock International Series in Pure and Applied Mathematics. McGraw-Hill
	Book Co., New York-Auckland-D\"usseldorf, third edition, 1976.
	
	\bibitem{SchnurrWoerner2011}
	A.~Schnurr and J.~H.~C. Woerner.
	\newblock Well-balanced {L}\'evy driven {O}rnstein-{U}hlenbeck processes.
	\newblock {\em Stat. Risk Model.}, 28(4):343--357, 2011.
	
	\bibitem{talarczyk2020}
	A.~Talarczyk and {\L}.~Treszczotko.
	\newblock Limit theorems for integrated trawl processes with symmetric
	{L}\'{e}vy bases.
	\newblock {\em Electron. J. Probab.}, 25:Paper No. 117, 24, 2020.
	
\end{thebibliography}

\end{document}